\documentclass[12pt,reqno]{article}
mm \textheight=8.9in

\usepackage{amssymb}
\usepackage{amsmath}
\usepackage{amsthm}
\usepackage{pb-diagram}
\usepackage{color}

\newcommand{\br}[3]{{$#1$}$\lower4pt\hbox{$\tp\atop\raise4pt \hbox{$\scriptscriptstyle{#2}$}$} ${$#3$}}
\newcommand{\tw}[3]{{$#1$}${\,\scriptscriptstyle {#2}}\atop\raise9pt\hbox{$\scriptstyle\tp$} ${$#3$}}
\newcommand{\ttps}[2]{{#1}\raise5pt\hbox{$\lower12pt\hbox{$\scriptstyle\tp$}\atop \lower0pt\hbox{$\tilde\;$}$}\raise4.5pt\hbox{${\scriptstyle{#2}}$}}
\newcommand{\st}[1]{\mbox{${\,\scriptscriptstyle {#1}}\atop\raise5.5pt\hbox{$*$}$}}

\newcommand{\rd}[1]{\mbox{${\,\scriptscriptstyle {#1}}\atop\raise5.5pt\hbox{$\bullet$}$}}
\newcommand{\rt}[1]{\otimes_\chi}
\newcommand{\lt}[1]{\mbox{${\,\scriptscriptstyle {#1}}\atop\raise5.5pt\hbox{$\ltimes$}$}}
\newcommand{\btr}{\raise1.2pt\hbox{$\scriptstyle\blacktriangleright$}\hspace{2pt}}
\newcommand{\btl}{\raise1.2pt\hbox{$\scriptstyle\blacktriangleleft$}\hspace{2pt}}

\newcommand{\lcr}{\raise1.0pt \hbox{${\scriptstyle\rightharpoonup}$}}
\newcommand{\rcr}{\raise1.0pt \hbox{${\scriptstyle\leftharpoonup}$}}

\newcommand{\ttp}{{\lower12pt\hbox{$\tp$}\atop \hbox{$\tilde\;$}}}

\newcommand{\gr}{ \mathrm{gr }}

\newcommand{\A}{{A}}

\newcommand{\Ru}{\mathcal{R}}

\newcommand{\Mc}{\mathcal{M}}

\newcommand{\Q}{\mathcal{Q}}

\newcommand{\C}{\mathbb{C}}

\newcommand{\tp}{\otimes}
\newcommand{\vt}{\vartheta}

\newcommand{\U}{U}

\newcommand{\ve}{\varepsilon}

\newcommand{\dt}{\delta}

\newcommand{\op}{\oplus}
\newcommand{\la}{\lambda}

\newcommand{\End}{\mathrm{End}}

\newcommand{\Span}{\mathrm{Span}}

\newcommand{\Tr}{\mathrm{Tr}}

\newcommand{\Rm}{\mathrm{R}}

\newcommand{\g}{\mathfrak{g}}
\renewcommand{\b}{\mathfrak{b}}

\newcommand{\h}{\mathfrak{h}}
\newcommand{\mub}{\boldsymbol{\mu}}

\newcommand{\nb}{\boldsymbol{n}}

\newcommand{\s}{\mathfrak{s}}
\renewcommand{\o}{\mathfrak{o}}

\newcommand{\eps}{\epsilon}

\newcommand{\nn}{\nonumber}
\newcommand{\p}{\mathfrak{p}}
\renewcommand{\l}{\mathfrak{l}}
\renewcommand{\c}{\mathfrak{c}}

\newcommand{\al}{\alpha}

\newcommand{\be}{\begin{eqnarray}}
\newcommand{\ee}{\end{eqnarray}}

\newtheorem{thm}{Theorem}[section]
\newtheorem{propn}[thm]{Proposition}
\newtheorem{lemma}[thm]{Lemma}
\newtheorem{corollary}[thm]{Corollary}

\newtheorem{definition}[thm]{Definition}

\newcount\prg

\newcommand{\parag}{\advance\prg by1 {\noindent\bf\thesection.\the\prg\hspace{6pt}}}

\begin{document}
\title{Quantization of borderline Levi conjugacy classes of orthogonal groups.}
\author{
Thomas Ashton and Andrey Mudrov\footnote{
This research is supported in part by the RFBR grant 12-01-00207a.} \vspace{20pt}\\
\small Department of Mathematics,\\ \small University of Leicester, \\
\small University Road,
LE1 7RH Leicester, UK\\
\small e-mail: am405@le.ac.uk\\
}

\date{}
\maketitle

\begin{abstract}
We construct equivariant quantization of a special family of Levi conjugacy classes
of the complex orthogonal group $SO(N)$, whose stabilizer
contains a Cartesian factor $SO(2)\times SO(P)$, $1\leqslant P<N$, $P\equiv N \mod 2$.
\end{abstract}


\section{Introduction}
This continuation of \cite{M1}  is devoted to equivariant quantization of a special family of conjugacy classes
in the complex algebraic group $G=SO(N)$. This work completes construction of quantum semisimple conjugacy classes of $SO(N)$
and, generally, of all simple groups of the infinite series.
Classes of our present concern have isotropy subgroups with
a Cartesian factor $SO(2)\times SO(P)$, where $P$ is of the same parity as $N$. Due to the isomorphism $GL(1)\simeq SO(2)$, they form a borderline between the Levi and non-Levi families, whose bulk cases have been processed
in \cite{M1,M2,M3}.

A solution of the classical Yang-Baxter equation makes $G$ a Poisson group with the Drinfeld-Sklyanin Poisson structure on it.
It also gives rise to another Poisson bracket on $G$ making it a Poisson manifold over $G$ with respect to the conjugacy transformation. This Poisson structure restricts to any conjugacy class of $G$. We construct a quantization of
its polynomial algebra along that structure, which is equivariant under the action of the quantized universal enveloping algebra
$U_q(\g)$. In the present paper we deal with the standard or Drinfeld-Jimbo classical $r$ matrix and the standard quantum group $U_q(\g)$. The constructed quantization can be automatically generalized for all other
factorizable $r$-matrices on $G$. For details, the reader is referred to  \cite{M1}.

Observe that semisimple conjugacy classes in $SO(N)$ can be categorized by their sets of eigenvalues: whether they include
both $\pm 1$ or not. The stabilizer subgroup of the second
type is Levi, and such a class is isomorphic to an adjoint orbit in $\s\o(N)$ as an affine variety.
Their quantization has been constructed in \cite{M2}.
The stabilizer of the first type contains a Cartesian factor
$SO(2m)\times SO(P)$, where $2m$ and $P$ are the multiplicities of the eigenvalues $-1$ and $+1$, respectively. If $m\geqslant 2$ (one should also assume $P\geqslant 4$ for even $N$),
the subgroup $L$ is not Levi. Such classes have been quantized in \cite{M1}. The remaining classes corresponding to $m=1$ form a special family, which was not covered before.



The quantization method  of the borderline Levi classes is similar to that used in \cite{M1} and \cite{M2}:
a realization of its quantized polynomial algebra in a $U_q(\g)$-module of highest weight.
In the case of interest, it is a parabolic Verma module of special weight. Due to this constrain, it
 is not a deformation of a Verma module over $U(\g)$. The boundary classes were not covered
in \cite{M2} because the analysis was based on the properties of the Shapovalov form derived by deformation
arguments from its classical counterpart. The specialization of the highest weight in our present approach
requires a special study of the module $\C^N\tp M_\la$ carried out in this paper.

Consider the borderline class $O$ passing through  the diagonal matrix $o$ with entries
$$\underbrace{\mu_1,\ldots, \mu_1}_{n_1},\ldots, \underbrace{\mu_\ell,\ldots, \mu_\ell}_{n_\ell}
,-1,
\underbrace{1,\ldots, 1}_{P},-1,
\underbrace{\mu_\ell^{-1},\ldots, \mu_\ell^{-1}}_{n_\ell},\ldots, \underbrace{\mu_1^{-1},\ldots, \mu_1^{-1}}_{n_1},
$$
where $P=2p$ if $N=2n$ and $P=2p+1$ if $N=2n+1$. The complex numbers $\{\mu_i\}_{i=1}^\ell$ and $\mu_{\ell+1}=-1$, $\mu_{\ell+2}=1$
satisfy the conditions $\mu_i\not =\mu_j^{\pm 1}$ for $i<j\leqslant \ell$ and $\mu_i^2\not= 1$ for  $1\leqslant  i\leqslant \ell$.
The centralizer of the point $o\in G$ is the subgroup
\be
L=GL(n_1)\times \ldots \times GL(n_\ell)\times SO(2)\times SO(P),
\label{gr_K}
\ee
whose Lie algebra $\l$ is a Levi subalgebra in $\g$,
$$
\l=\g\l(n_1)\oplus\cdots\g\l(n_\ell)\oplus \s\o(2)\oplus \s\o(P).
$$
The subgroup $L$ is determined by  an integer valued vector $\nb=(n_i)_{i=1}^{\ell+2}$ subject to $\sum_{i=1}^{\ell+2}n_i=n$.
We reserve the integer $l$ for $\sum_{i=1}^\ell n_i$, so that $l+1+p=n$. Here $n_{\ell+1}=1$ and $n_{\ell+2}=p$.
Let $\Mc_L$ denote the moduli space of conjugacy classes with
the fixed isotropy subgroup (\ref{gr_K}), regarded as Poisson spaces as fixed in \cite{M1}. We introduce
the subspace $\Mc'_L$ of classes with $\mu_{\ell+1}=-1$.
The sets of all  $\ell+2$-tuples $\mub$  as specified above
parameterize $\Mc_L$ and $\Mc_L'$ although not uniquely. We denote these sets by $\hat \Mc_L$ and, respectively, $\hat \Mc'_L$.

As a variety, the class $O$ associated with $\mub$ and $\nb$ is determined  by the set of equations
\be
(A-\mu_1)\ldots (A-\mu_\ell)
(A+1)(A-1)(A-\mu_\ell^{-1})\ldots (A-\mu_1^{-1})=0,
\label{min_pol_cl}
\\
\Tr(A^k)=\sum_{i=1}^\ell n_i(\mu_i^{k}+\mu_i^{-k})+2(-1)^k+P, \quad k=1,\ldots, N,
\label{tr_cl}
\ee
where the matrix multiplication in the first line is understood. This system is polynomial
in the matrix entries $A_{ij}$ and defines an ideal of $\C[\End(\C^N)]$ vanishing on $O$.
\begin{thm}
The system of polynomial relations (\ref{min_pol_cl}) and (\ref{tr_cl}) generates the defining ideal of the
class $O$ in $\C[SO(N)]$.
\label{prop_clas_so}
\end{thm}
\begin{proof}
The proof is similar to \cite{M3}, Theorem 2.3.
\end{proof}
\noindent
Our goal is a generalization of this statement for the quantized polynomial algebra of  $O$.
\section{Parabolic Verma module $M_\la$}
We adopt certain  conventions concerning  representations of quantum groups, which are similar to \cite{M1}.
Unless otherwise stated, the quantum group $U_q(\g)$ and its modules are considered over the complex field,
 upon specialization of $q$ to not a root of unit. Extension of the ring of scalars via  $q=e^\hbar$
determines the embedding  $U_q(\g)\subset  U_\hbar(\g)$, where the latter is considered over
the ring $\C[\![\hbar ]\!]$ of formal power series in $\hbar$.
We assume that $U_\hbar(\g)$-modules are free over $\C[\![\hbar]\!]$ and their rank will
be referred to as  dimension.
Finite dimensional $U_\hbar(\g)$-modules are deformations of their classical counterparts, and we drop the reference
to $\hbar$ to simplify notation. For instance, the natural
$N$-dimensional representation of $U_\hbar(\g)$  will be denoted simply by $\C^N$.

Let $U_\hbar(\h)$ be the Cartan subalgebra in $U_\hbar(\g)$. We shall deal with  $U_\hbar(\h)$-diagonalizable, i.e. weight modules. If $V$ is an $\h$-invariant subspace, we mean by $[V]_\al$
the subspace of weight $\al\in \h^*$.
We stick to  the additive parametrization of weights facilitated by the embedding $U_q(\h)\subset U_\hbar(\h)$.
Under this convention, weights belong to  $\frac{1}{\hbar} \h^*[\![\hbar]\!]$ and are well defined on
$t^{\pm 1}_{\al_i}\in q^{\h}$. It is
sufficient for our needs to confine them to the subspace $ \hbar^{-1}\h^*\op \h^* \subset \hbar^{-1}\h^*[\![\hbar]\!]$.

We denote by $\c_\l\subset \h$ the center of $\l$ and realize its dual $\c_\l^*$ as a subspace in $\h^*$ thanks to
the canonical inner product.
Let $\p^+=\l+\g_+\subset \g$ denote the parabolic subalgebra, where $\g_\pm$ are the subalgebras generated by the positive and negative Chevalley  generators.
An element $\la\in \mathfrak{C}^*_{\l}= \hbar^{-1} \c_\l^*\op  \c_\l^*$ defines
a one-dimensional representation of $U_q(\l)$ denoted by  $\C_\la$. Its restriction to $U_q(\h)$ acts by the
assignment $q^{\pm h_\al}\mapsto q^{\pm (\al,\la)}$, $\al \in \Pi^+$. Since $q=e^\hbar$, the pole in $\la$ is compensated, and the representation is correctly defined.
It extends to $U_q(\p^+)$ by setting it zero on $\g_+\subset \p_\l^+$. Denote by $M_\la$ the parabolic Verma module $U_q(\g)\tp_{U_q(\p^+)}\C_\la$, \cite{Ja}. Regarded as a $U_q(\g_-)$-module by restriction from $U_q(\g)$, $M_\la$ is isomorphic to the quotient $U_q(\g_-)/U_q(\g_-)\l_-$, which we denote by $U_\l^-$.

The vector space $\C^N$ is regarded as a $U_q(\g)$-module supporting its natural representation. Of key importance for us is the
structure of the tensor product $\C^N\tp M_\la$. The element $\Ru_{12}\Ru$ expressed through the universal R-matrix $\Ru\in U_\hbar(\g)\tp U_\hbar(\g)$ operates on $\C^N\tp M_\la$ as an invariant matrix
$\Q\in \End(\C^N)\tp U_q(\g)$, which commutes with $\Delta(x)$ for all $x\in U_q(\g)$. The normal form of 
$\Q$ is determined by the submodule structure of $\C^N\tp M_\la$, the study of which takes the majority of this paper.
The eigenvalues of $\Q$ are found in \cite{M2}. It is also known that $\Q$ is semisimple for generic $\la\in \mathfrak{C}^*_{\l}$.
Then we are going to check that  $\Q$ remains semisimple  for a certain set of  $\la$ of our interest.

Let $\{\ve_i\}_{i=1}^N$  be the weights of the natural $U_q(\g)$-module $\C^N$. Then $\{\ve_i\}_{i=1}^n$,
$n=[\frac{N}{2}]$ (the integer part of $\frac{N}{2}$), form an orthogonal basis in $\h^*$, and $\ve_i=-\ve_{N+1-i}$.
The positive roots are expressed through $\{\ve_i\}_{i=1}^n$ in the standard way as fixed in \cite{M1}.
Denote by $w_i\in \C^N$ the standard  basis elements of weight $\ve_i$, $i=1,\ldots, N$.
The natural $U_q(\g)$-module splits into irreducible $U_q(\l)$-modules,
\be
\C^N=(\C^{n_1}\oplus\cdots \oplus\C^{n_{\ell}})\oplus\C\oplus \C^P \oplus\C\oplus (\C^{n_{\ell}} \oplus\cdots \oplus \C^{n_1}),
\label{l-decomp}
\ee
which decomposition is compatible with the basis $\{w_i\}_{i=1}^N=\cup_{i=1}^{2\ell+3}\{w_{k}\}_{k=m_i}^{m_i-1}$ counting from the left.
Here $m_i=n_1+\cdots+n_{i-1}+1$ for $i=1,\ldots, \ell+2$, and
$m_{2\ell+4-i}=N+1-\sum_{k=1}^in_k$, $i=1,\ldots, \ell$. Note that  $w_{m_i}$ is the highest weight vector of the corresponding irreducible $\l$-submodule in $\C^N$.

For $\la\in \mathfrak{C}^*_{\l}$, the operator  $\Q\in \End(\C^{N}\tp  M_\la)$ satisfies the equation $\prod_{i=1}^{2\ell+3}(\Q-x_i)=0$ with the roots
\be
\begin{array}{ccc}
x_i&=&q^{2(\la,\ve_{m_i})-2(m_i-1)},\> i=1,\ldots,\ell+2,
\\
x_{2\ell+4-i}&=&q^{-2(\la,\ve_{m_i})-2N+2(m_i+n_i)},
\>i=1,\ldots,\ell+1,
\end{array}
\label{eigenvalues}
\ee
see \cite{M2}, Theorem 4.2.
The root $x_i$ corresponds to a submodule $M_i\subset \C^N\tp M_\la$, where $\Q$
acts as multiplication by $x_i$. For generic $\la\in \mathfrak{C}^*_{\l}$ and $q$,
the roots $x_i$ are pairwise distinct, and $\C^N\tp M_\la=\oplus_{i=1}^{2\ell+3}M_i$.

In this paper, we are interested in special $\la$ making $x_{\ell+1}=q^{2(\la,\ve_{l+1})-2l}$  equal to $x_{\ell+3}=q^{-2(\la,\ve_{l+1})-2l-2P}$.
In particular, this condition is satisfied if
\be
q^{2(\la, \ve_{l+1})}=-q^{-P}.
\label{-1eig}
\ee
Let $\mathfrak{C}^*_{\l,'}$ be the subset  of all weights $\la\in \mathfrak{C}^*_{\l}$ subject to (\ref{-1eig}).
We prove that, for generic $\la\in \mathfrak{C}^*_{\l,'}$ and generic $q$ including $q\to 1$, the direct sum decomposition
of $\C^N\tp M_\la$ still holds,
and the operator $\Q$ is  semisimple. To this end, we study the submodules $M_{\ell+1}$ and $M_{\ell+3}$ and show that
 their sum is direct for all $\la$ satisfying (\ref{-1eig}). Our analysis is based on
calculation of singular vectors generating $M_{\ell+1}$ and $M_{\ell+3}$.

As in \cite{M2}, we introduce a subspace of weights   that we use for the parametrization of
$\Mc_{L}'$, the moduli space  of borderline conjugacy classes with fixed $L$.
Put $\mu^0_k=e^{2(\la,\ve_{m_k})}$, for $k=1,\ldots, \ell+2$. The subset $\c_{\l,'}^*\subset \c_{\l}^*$
is specified by the condition $\mu^0_{\ell+1}=-1$.
Let $\c_{\l,reg}^*$ denote the set of all weights $\la\in \c_\l^*$
such that $\mub^0\in \hat \Mc_L$ and similarly define ${\c}_{\l,reg'}^*\subset {\c}_{\l,'}^*$ by the requirement  $\mub^0\in \hat \Mc_L'$.
Finally, we introduce $\mathfrak{C}^*_{\l,reg'}=\mathfrak{C}^*_{\l,'}\cap (\hbar^{-1}{\c}_{\l,reg'}^*\oplus {\c}_{\l}^*) $.
The subset  $\mathfrak{C}^*_{\l,reg'}$ is dense in $\mathfrak{C}^*_{\l,'}$.

rluxemburg21@mail.ru

\section{On singular vectors in $\C^N\tp M_\la$}
In this section, $\l$ is the Levi subalgebra $\h+\s\o(P)$, which can be otherwise put as $\ell=l$. The parabolic Verma module $M_\la$ is relative
to this subalgebra. In other words, $\la\in \mathfrak{C}^*_{\l}$ if and only if
$(\la, \ve_i)=0$, $i=l+2,\ldots, n$.

Given weight $\la \in \frac{1}{\hbar}\h^*\oplus \h^*$ we denote $\la_i=(\la,\ve_i)$, for all $i=1,\ldots,N$.
The natural representation of $U_q(\g)$ on $\C^N$ is determined by the action
$f_{\ve_j-\ve_k} w_i=(-1)^{\eps_i}\dt_{ij}w_k$,  $e_{\ve_j-\ve_k} w_i=(-1)^{\eps_j}\dt_{ki}w_j$,
for $\ve_j-\ve_k \in \Pi^+$,  where  $\eps_i=0$ if $i\leqslant \frac{N}{2}$ and
$\eps_i=1$ otherwise. Note that Chevalley generators are normalized so that their representation matrices
are independent of $q$.

For $\g=\s\o(2n+1)$, the natural representation is determined, up to scalar multipliers, by the graph
\begin{center}
\begin{picture}(390,45)
\put(0,0){$w_{2n+1}$}
\put(10,15){\circle{3}}
\put(45,15){\vector(-1,0){30}}
\put(50,15){\circle{3}}
\put(100,15){$\ldots$}
\put(85,15){\vector(-1,0){30}}
\put(45,0){$w_{2n}$}
\put(160,15){\circle{3}}
\put(155,15){\vector(-1,0){30}}
\put(150,0){$w_{n+2}$}
\put(195,15){\vector(-1,0){30}}
\put(200,15){\circle{3}}
\put(190,0){$w_{n+1}$}
\put(235,15){\vector(-1,0){30}}
\put(240,15){\circle{3}}
\put(230,0){$w_{n}$}
\put(275,15){\vector(-1,0){30}}
\put(345,15){\vector(-1,0){30}}
\put(350,15){\circle{3}}
\put(345,0){$w_{2}$}
\put(290,15){$\ldots$}
\put(385,15){\vector(-1,0){30}}
\put(390,15){\circle{3}}
\put(385,0){$w_{1}$}
\put(25,20){$f_{\al_1}$}
\put(65,20){$f_{\al_2}$}
\put(135,20){$f_{\al_{n-1}}$}
\put(175,20){$f_{\al_{n}}$}
\put(215,20){$f_{\al_{n}}$}
\put(255,20){$f_{\al_{n-1}}$}
\put(325,20){$f_{\al_{2}}$}
\put(365,20){$f_{\al_{1}}$}

\end{picture}
\end{center}
The graph for $\g=\s\o(2n)$ is
\begin{center}
\begin{picture}(440,60)
\put(0,0){$w_{2n}$}
\put(10,15){\circle{3}}
\put(50,15){\circle{3}}
\put(45,15){\vector(-1,0){30}}
\put(85,15){\vector(-1,0){30}}
\put(100,15){$\ldots$}
\put(155,15){\vector(-1,0){30}}
\put(45,0){$w_{2n-1}$}
\put(160,15){\circle{3}}
\put(150,0){$w_{n+2}$}
\put(195,15){\vector(-1,0){30}}
\put(200,15){\circle{3}}
\put(190,0){$w_{n+1}$}
\qbezier(163,22)(200,48)(237,22) \put(162,21.5){\vector(-2,-1){2}}
\qbezier(203,22)(240,48)(277,22) \put(202,21.5){\vector(-2,-1){2}}
\put(240,15){\circle{3}}
\put(230,0){$w_{n}$}
\put(275,15){\vector(-1,0){30}}
\put(280,15){\circle{3}}
\put(270,0){$w_{n-1}$}
\put(330,15){$\ldots$}
\put(315,15){\vector(-1,0){30}}
\put(385,15){\vector(-1,0){30}}
\put(390,15){\circle{3}}
\put(385,0){$w_{2}$}
\put(425,15){\vector(-1,0){30}}
\put(430,15){\circle{3}}
\put(425,0){$w_{1}$}
\put(25,20){$f_{\al_1}$}
\put(65,20){$f_{\al_2}$}
\put(135,20){$f_{\al_{n-2}}$}
\put(172,20){$f_{\al_{n-1}}$}
\put(195,38){$f_{\al_{n}}$}
\put(252,20){$f_{\al_{n-1}}$}
\put(290,20){$f_{\al_{n-2}}$}
\put(235,38){$f_{\al_{n}}$}
\put(365,20){$f_{\al_{2}}$}
\put(405,20){$f_{\al_{1}}$}
\end{picture}
\end{center}
Reversing the arrows one gets the graphs for $e_\al$, $\al\in \Pi^+$.

Similarly, one can consider dual natural representation of $U_q(\g)$ on $\C^N$. In the dual basis $\{v^i\}_{i=1}^N$, the graphs will be similar,
with all arrows reversed.

Suppose that there is path from the node $w_i$ to $w_j$ on the representation graph. Then there is a Chevalley monomial $\psi\in U_q(\g_-)$ such that
$w_j$ is equal to $\psi w_i$, up to an invertible scalar multiplier.
Such $\psi$ is unique, which is obvious for odd $N$ and still true for even $N$, since
$f_{\al_n}f_{\al_{n-1}}=f_{\al_{n-1}}f_{\al_n}$. We denote this monomial $\psi_{ji}$ and
write $i\prec j$. This makes the integer interval $[1,N]$ a poset with the Hasse diagram above.

In what follows, we also use the monomial  $\psi^{ij}$ obtained from  $\psi_{ji}$ by reversing the order of Chevalley generators, so
that $v^i=\psi^{ij}v^j$. We also put $\psi^{ii}=1$ for all $i$.
It is clear that $\psi^{ij}=\psi^{im}\psi^{mj}$ for any $m$ such that $i\preceq m\preceq j$.
\begin{definition}
We call $\psi^{ij}$, $i\preceq j$, the principal  monomial of weight $\ve_j-\ve_i$.
\end{definition}
\noindent 
Remark that all Chevalley monomials of weight $\ve_j-\ve_i$ are obtained from $\psi^{ij}$ by permutation of factors.

Recall that a non-zero weight vector $v$ in a $U_q(\g)$-module is called singular if it generates  the trivial $U_q(\g_+)$-submodule, i.e. $e_\al v=0$, for all $\al \in \Pi^+$. Since the weights of $e_\al v$ are pairwise distinct, this is equivalent
to the equation $E v=0$, where $E=\sum_{m=1}^n e_{\al_m}$. We will also work with
the operator $E'=\sum_{m=2}^n e_{\al_m}$, in view of Corollary \ref{generating_coeff} below.

\begin{lemma}
\label{dual-contragredient}
Let $W$ be a finite dimensional $U_q(\g)$-module and $W^*$ its right dual module. Let $Y$ be a $U_q(\g)$-module. Singular vectors in $W\tp Y$
are parameterized by homomorphisms  $W^*\to Y$ of $U_q(\g_+)$-modules.
\end{lemma}
\begin{proof}
Choose a weight basis $\{w_i\}_{i=1}^d\subset W$,  where $d=\dim W$. Suppose that
$u\in W\tp Y$ is a singular vector,
$
u=\sum_{i=1}^d w_i\tp y_i,
$ for some
$y_i\in Y$. Let $\pi\colon U_q(\g)\to \End(W)$ denote the representation homomorphism, $\pi(u)w_i=\sum_{j=1}^N \pi(u)_{ij}w_j$.
We have, for $\al \in \Pi^+$,
\be
e_\al u=\sum_{i=1}^d \sum_{j=1}^d \pi(e_\al)_{ij}w_j\tp y_i+\sum_{i=1}^d q^{(\al,\ve_i)}w_i\tp e_\al y_i.
\label{y_i conat_fact}
\ee
So $e_\al  u=0$ is equivalent to
$
e_\al y_i=-q^{-(\al,\ve_i)}\sum_{j=1}^d  \pi(e_\al)_{ji}y_j.
$
The vector space $\Span\{y_i\}_{i=1}^d$ supports the right dual representation of $U_q(\g_+)$, provided
$y_i$ are linear independent. In general, it is a quotient of the right dual representation.
\end{proof}
Formula (\ref{y_i conat_fact}) can be more explicitly rewritten as
$$
y_j=(-1)^{\eps_i+1}q^{(\ve_i-\ve_j,\ve_i)}e_{\ve_i-\ve_j}y_i
$$
for all $i,j\in [1,N]$ such that $\ve_i-\ve_j\in \Pi^+$. In the following corollary, $M_\la$ is a parabolic Verma module
relative to arbitrary $\l$.

\begin{corollary}
\label{generating_coeff}
Singular vectors $\{u_i\}\in \C^N\tp M_\la$ are parameterized by weight elements $y\in M_\la$ satisfying
$e_{\al_1}^3y=0$ if $N=3$, $e_{\al_1}^2y=e_{\al_2}^2y=0$ if $N=4$ and $e_{\al_1}^2 y= E'y=0$ for $N>4$.
\end{corollary}
\begin{proof}
The weight $\ve_1$ is integral dominant.
The dual natural representation of $U_q(\g)$ is generated by the vector of lowest weight $-\ve_1$.
When restricted to $U_q(\g_+)$,
it is isomorphic to a quotient of the left regular $U_q(\g_+)$-module. It is the quotient by the left ideal in $U_q(\g_+)$
 generated
by $e_{\al_1}^3$ if $N=3$, by $e_{\al_1}^2,e_{\al_2}^2$ if $N=4$, and  by  $e_{\al_1}^2, e_{\al_i}$, $i=2,\ldots,n$ if $N>4$. Therefore, all homomorphisms
from the co-natural module to $M_\la$ are generated by the assignment $U_q(\g_+)\ni 1\to y\in M_\la$,
where $y$ satisfies the hypothesis.
\end{proof}
Singular vectors generate $U_q(\l)$-submodules of highest weight.
It is known that, for generic $\la$,  singular vectors in $\C^N\tp M_\la$ are parameterized by the highest weights $\nu$ of the
irreducible $U_q(\l)$-submodules in $\C^N$ and carry the weights $\la+\nu$.
We denote by $u_j$ the singular vector of weight $\la+\ve_j$, $j=1,\ldots, N$, which is defined up to a non-zero scalar factor.
We can write
$$
u_j=\sum_{j=1}^{N-l}w_{j}\tp y_{j,i}, \quad j=1,\ldots, N,
$$
where $y_j\in M_\la$ is an element of weight $\la+\ve_j-\ve_i$, $i\leqslant j$. For each $j$ the linear span $\{y_{j,i}\}_{i=1}^j$
supports a quotient of the co-natural representation of $U_q(\g_+)$, which is cyclicly generated by $\{y_{j,1}\}$

Singular vectors $u_{i}$, $i=1,\ldots, n-1$, are related to the subalgebra $\g\l(n)\subset \g$ and  can be found in
\cite{AM}. Singular vector $u_{n+1}$ in the case of $\g=\s\o(2n)$ is related to another copy of $\g\l(n)$ with
$\al_{n-1}$ replaced by $\al_{n}$. Singular vector $u_{n+1}$ for $\g=\s\o(2n+1)$ can be constructed as follows.
Define the "dynamical root vectors" $f_{\ve_k}$ by setting $f_{\ve_n}= f_{\al_n}$ and
$$
 \quad
f_{\ve_{k-1}}=f_{\al_{k-1}}f_{\ve_k}[h_{\ve_k}+n-k+1]_q- f_{\ve_k}f_{\al_{k-1}}[h_{\ve_k}+n-k]_q
$$
for all $k=n-1,\ldots,1$. It is also convenient to put  $f_{0}=1$ taking into account $\ve_{n+1}=0$. Let $M_\la$ be a Verma module and $v_\la$ its canonical generator.
One can check  the identity
$$
e_{\al_k}f_{\ve_i}v_{\la}=\dt_{ki}[\la_{i}+n-i]f_{\ve_{i+1}}v_{\la},
$$
by induction on $i$.
Setting $y_{n+1,1}=f_{\ve_{1}}v_\la$, one obtains
$
y_{n+1,i}=(-q)^{i-1}\prod_{k=1}^{i-1}[\la_{i}+n-k]f_{\ve_{i}}v_\la,
$
 $i=1,\ldots n+1$.

We need not know all singular vectors for the purpose of this study.
We are especially interested in  $u_{N-l}$ carrying the weight $\la-\ve_{l+1}$. It  is expanded over the basis
$\{w_i\}_{i=1}^N\subset \C^N$ as
$
u_{N-l}=\sum_{i=1}^{N-l}w_{i}\tp y_i
$
with coefficients $y_i=y_{N-l,i}$ of weight $\la-\ve_i-\ve_{l+1}$, $i=1,\ldots, l+1$. They are generated by
$y_1$ via the co-natural action of $U_q(\g_+)$. We call $y_1$ the generating coefficient.
Our next goal is to evaluate $y_1$.

Consider the graph corresponding to the co-natural representation of $U_q(\g_+)$ for $N>3$.
\begin{center}
\begin{picture}(310,35)
\put(0,0){$y_{N-l}$}
\put(10,15){\circle{3}}
\put(45,15){\vector(-1,0){30}}
\put(60,15){$\ldots$}
\put(115,15){\vector(-1,0){30}}
\put(120,15){\circle{3}}
\put(110,0){$y_{l+1}$}
\put(155,15){\vector(-1,0){30}}
\put(160,15){\circle{3}}
\put(155,0){$y_{l}$}
\put(195,15){\vector(-1,0){30}}
\put(265,15){\vector(-1,0){30}}
\put(270,15){\circle{3}}
\put(265,0){$y_{2}$}
\put(210,15){$\ldots$}
\put(305,15){\vector(-1,0){30}}
\put(310,15){\circle{3}}
\put(305,0){$y_{1}$}
\put(25,20){$e_{\al_{l+1}}$}
\put(95,20){$e_{\al_{l+1}}$}
\put(135,20){$e_{\al_{l}}$}
\put(175,20){$e_{\al_{l-1}}$}
\put(245,20){$e_{\al_{2}}$}
\put(285,20){$e_{\al_{1}}$}

\end{picture}
\end{center}
One can readily write down $y_{i}$ for $l+2\leqslant i\leqslant N-l$, up to a scalar factor.
Indeed, the corresponding weight spaces in $M_\la$ have dimension $1$. Suppose that $\psi^{i,N-l}=f_\al\psi^{j,N-l}$
for  $\al=\ve_i-\ve_j \in \Pi^+$ (for odd $N$,  $j$ is  always $i-1$, while for even $N$  $j$ may be also $i-2$ for $i=n+1, n+2$).
Then $e_\al \psi^{i,N-l}v_\la\sim\psi^{i+1,N-l}v_\la$ and $y_{i}\sim \psi^{i,N-l} v_\la$:
$$y_{N-l}\sim v_\la, \quad y_{N-l-1}\sim f_{l+1}v_{\la}, \quad y_{N-l-2}\sim f_{l+2}f_{l+1}v_{\la},\quad\ldots
$$
In particular,
$y_{l+1}\sim f_{l+1}\stackrel{<}{\ldots} f_{n-1}f_nf_nf_{n-1}\stackrel{>}{\ldots} f_{l+1}v_\la$
for  odd $N$ and a similar expression with $f_{n-1}f_n$ in place of $f_nf_n$ for even $N$.

 The problem essentially boils down to finding $y_i$ with $i\leqslant l+1$. These coefficients feature
the following chain property. Let $\g_{i}'\subset \g$ denote the subalgebra with simple roots $\{\al_j\}_{j=i}^n$ and
let $M'_{i,\la}\subset M_\la$ be the $U_q(\g_{i}')$-submodule generated by $v_\la$. If  $y_i\in M'_{i,\la}$, then $y_i$ is the generating
coefficient for a $U_q(\g_{i}')$-singular vector in $\C^{N-2i+2}\tp M'_{i,\la}$, as follows from the representation graph. This observation enables  construction of $y_i$  by descending induction starting from $y_{l+1}\in M'_{l+1,\la}$, which is done in the next section.

\subsection{Symmetric classes}
\label{Sec_l=0}
In this section, we fix $l=0$ or equivalently  $n=1+p$. This assumption corresponds to
the symmetric conjugacy class of matrices with eigenvalues
$-1$ and $+1$ of multiplicities $2$ and $P$, respectively.
The singular vector of interest has weight $\la+\ve_{l+1}=\la-\ve_{1}$.

We introduce the following basis in $[U_\l^-]_{-2\ve_1}$. Observe  that $d^0_{P}=\dim [U_\l^-]_{-2\ve_1}$
is $p+1$ for odd $P$ and $p$ for even $P$ (recall that $P\equiv N \mod 2$ is the multiplicity of $+1$ in the spectrum of the conjugacy class).
Define monomials $\phi_m$, $m=1,\ldots, d^0_P$, by
\be
\phi_{m} &=&
\left\{
\begin{array}{ll}
\begin{array}{lll}
f_{\al_m} \overset{>}{\ldots} f_{\al_1} f_{\al_{m+1}}\overset{>}{\ldots} f_{\al_{p+1}}f_{\al_{p+1}}\overset{>}{\ldots} f_{\al_1},&&1\leqslant m\leqslant p+1\\
\end{array}& \mbox{for odd } N,\\
\left.
\begin{array}{ll}
f_{\al_m} \overset{>}{\ldots} f_{\al_1} f_{\al_{m+1}}\overset{<}{\ldots} f_{\al_{p-1}}f_{\al_p}f_{\al_{p+1}}f_{\al_{p-1}}\overset{>}{\ldots} f_{\al_1},&1\leqslant m\leqslant p-1\\
f_{\al_p}f_{\al_{p-1}} \overset{>}{\ldots} f_{\al_1} f_{\al_{p+1}}f_{\al_{p-1}}\overset{>}{\ldots} f_{\al_1},& m=p\\
f_{\al_{p+1}}f_{\al_{p-1}} \overset{>}{\ldots} f_{\al_1} f_{\al_p}f_{{\al_p-1}}\overset{>}{\ldots} f_{\al_1},& m=p+1
\end{array}\right\}&\mbox{ for even } N.
\end{array}
\right.
\nn
\ee
All $\phi_m$ have weight $-2\ve_1$.
Using the Serre relations, one can check for even $N$ that $\phi_{p+1} = f_{\al_p}f_{\al_{p-1}} \overset{>}{\ldots} f_{\al_1} f_{\al_{p+1}}f_{\al_{p-1}}\overset{>}{\ldots} f_{\al_1}=\phi_{p}$, so the number of independent $\phi_m$ is 
equal to $d^0_P=\dim [U^-_\l]_{-2\ve_1}$. Still
it is convenient to consider both $\phi_{p}$ and $\phi_{p+1}$.

The leftmost position in all $\phi_m$ is occupied by $f_{\al_m}$. We  define
vectors $\phi_m'$ of weight $-2\ve_1+\al_m$ obtained from $\phi_m$ by deleting this $f_{\al_m}$:
$$
\phi_{m}' = f_{\al_{m-1}} \overset{>}{\ldots} f_{\al_1} f_{\al_{m+1}}\overset{<}{\ldots} f_{\al_{p+1}}f_{\al_{p+1}}\overset{>}{\ldots} f_{\al_1}
\quad \mbox{for odd } N, \mbox{ and  }
$$
$$
\phi_{m}' =
\left\{
\begin{array}{ll}
f_{\al_{m-1}} \overset{>}{\ldots} f_{\al_1} f_{\al_{m+1}}\overset{<}{\ldots} f_{\al_{p-1}}f_{\al_p}f_{\al_{p+1}}f_{\al_{p-1}}\overset{>}{\ldots} f_{{\al_1}},& m\leqslant p-1\\
f_{\al_{p-1}} \overset{>}{\ldots} f_{\al_1} f_{\al_{p+1}}f_{\al_{p-1}}\overset{>}{\ldots} f_{\al_1},& m=p\\
f_{\al_{p-1}} \overset{>}{\ldots} f_{\al_1} f_{\al_p}f_{\al_{p-1}}\overset{>}{\ldots} f_{\al_1},& m=p+1
\end{array}
\right. ,
\quad \mbox{for even } N.
$$
Abusing notation, we will also identify $\phi_m$ and $\phi'_m$ with their images in the quotient $U_\l^-$. 
\begin{lemma}
\label{phi'}
The monomial $\phi_{m}'$ spans $[U_\l^-]_{-2\ve_1+\al_m}$ for each $m=1,\ldots, p+1$.
\end{lemma}
\begin{proof}
One can check that $\dim[U_\l^-]_{-2\ve_1+\al_m}=1$, so to prove the statement, we must prove that $\phi'_m\not=0$.
The squared norm $\langle\phi'_m v_\la,\phi'_m v_\la\rangle$
with respect to the Shapovalov form on $M_\la$
is equal to $[\la_{1}]_{q}$ for $m=1$ and to $[\la_{1}]_{q}[\la_{1}-1]_{q}$ otherwise.
It is not zero if  $[\la_{1}]_{q}[\la_{1}-1]_{q}\not=0$. Due to the isomorphism $M_\la\simeq U_\l^-$,
 $\phi'_m\not =0$ as well as its projection in $U_\l^-$ for generic $\la$. But $\phi'_m$ is independent of $\la$, which
 completes the proof.
\end{proof}
Let $\Phi^{0}$ denotes the linear span of $\{\phi_m v_\la\}_{m=1}^{d^0_P}\subset M_\la$.
Denote by  $\hat E$ the composition $\C^{d^0_P}\to \Phi^{0}\to M_\la$ of linear maps,
$(A_m)\mapsto \sum_{m=1}^{d^0_P} A_m \phi_m =y\mapsto E y$.
For $N\geqslant 5$, the operator $\hat E$ acts on
$\C^{d^0_P}$ by $(A_m)_{m=1}^{d^0_P}\mapsto \sum_{m=1}^{p+1}B_m\phi_m'v_\la$,
where the scalar coefficients $B_m$ are given in Appendix A.

\begin{lemma}
\label{Einjectve}
Suppose that $N\geqslant 5$. Then the map $\hat E$ is injective  for generic $\la$.
\end{lemma}
\begin{proof}
Define
\be
A_{m}&=&\left( -1 \right)^{m-1} \left[ \frac{P}{2}-m +1\right]_{q}, \quad m=1,\ldots,d^0_P.
\label{A_level_1}
\ee
For $N\geqslant 5$, one can check that (\ref{A_level_1}) is a unique solution of the system of equations $B_i=0$, $i=2,\ldots,p+1$,
up to a  common  scalar factor.
This makes $B_1=A_1[\la_1]_q+A_2[\la_1-1]_q$ into $[\la_1+\frac{P}{2}-1]_q$, which does not vanish for generic
$\la$.
\end{proof}

\begin{corollary}
a) The system $\{\phi_m v_\la\}_{m=1}^{d^0_P}$ forms a basis in $[M_\la]_{\la-2\ve_1}$. b) The vector
 $f_{2\ve_{1}}^{(P)}v_\la=\sum_{m=1}^{d^0_P} A_{m}\phi_{m}v_\la$, where $A_m$ are given by (\ref{A_level_1}),
 is a generating coefficient. c) It is a unique generating coefficient of weight $\la-2\ve_1$, up to
 a scalar factor.
\end{corollary}
\begin{proof}
The statement is obvious for $N=3,4$ with $p=0$ and, respectively, $p=1$. Then $d^0_P=1$ and the vectors   $f_{2\ve_{1}}^{(1)}v_\la=[\frac{1}{2}]_qf_{\al_1}^2v_\la$,  $f_{2\ve_{1}}^{(2)}v_\la=f_{\al_1}f_{\al_2}v_\la$
satisfy the conditions $e_{\al_1}^3f_{2\ve_{1}}^{(1)}v_\la=0$ and $e_{\al_1}^2f_{2\ve_{1}}^{(2)}v_\la=e_{\al_2}^2f_{2\ve_{1}}^{(2)}v_\la=0$,
as required.

Now suppose that $N\geqslant 5$. Since the operator $\hat E$ is injective, the map $\C^{d^0_P}\to \Phi^{0}$ is
injective too. It is surjective by construction, hence it is a bijection.
For generic $\la$, the vectors $\{\phi_mv_\la\}^{d^0_P}_{m=1}$ form a basis in $\Phi^{0}$ and hence in $[M_\la]_{\la-2\ve_1}$, as the latter has dimension $d^0_P$. The vectors $\{\phi_m\}^{d^0_P}_{m=1}$ form a basis in $[U_\l^-]_{-2\ve_1}$, due to the
linear isomorphism  $[M_\la]_\mu\simeq [U_\l^-]_{\mu-\la}$. These vectors are independent of $\la$, hence they form a
basis at all $\la$, as well as $\{\phi_mv_\la\}^{d^0_P}_{m=1}$.
This implies that $f_{2\ve_{1}}^{(P)}v_\la\not =0$, and it is a unique generating coefficient, up to a scalar factor.
\end{proof}
\subsection{The case $l=1$}
\label{Sec_l=1}
To keep reference to the symmetric case considered in the previous section, we enumerate the simple roots $\Pi_\g=\{\al_i\}_{i=0}^{p+1}$. Then the roots
$\{\al_i\}_{i=1}^{p+1}$ correspond to the subalgebra $U_q(\g'_1)\subset U_q(\g)$. Under this embedding,
we regard $\phi_m$ and $f_{2\ve_{1}}^{(P)}$ constructed in the previous section as elements of $U_q(\g)$.

Observe that $d_{P}^{1}=\dim [M_\la]_{\la-\ve_0-\ve_1}$ is equal to $3p+3$ for odd $N$ and $3p+1$ for even $N$.
The only generator which does not commute with $f_{\al_0}$ is $f_{\al_1}$, and it enters $\phi_m$ twice.
There are three possible ways to allocate $f_{\al_0}$ relative to these $f_{\al_1}$. We use this observation to construct
the basis in $[U^-_\l]_{-\ve_0-\ve_1}$ from the basis in $[U^-_\l]_{-2\ve_1}$.
For all $m = 1, \ldots , p+1$, define $\phi^1_m= f_{\al_0}\phi_m$ and $\phi^3_m= \phi_m f_{\al_0}$.
Define $\phi_m^2$ to be the monomial obtained from
$\phi_m$ by replacing the rightmost copy of $f_{\al_1}$ with $f_{\al_0}f_{\al_1}$.
For even $N$, the equality $\phi_{p+1}=\phi_{p}$ implies $\phi_{p+1}^1=\phi^1_{p}$ and $\phi_{p+1}^3=\phi^3_{p}$,
so we have effectively  $3p+1$ monomials for even $N$ and $3p+3$ monomials for odd $N$.

As in the symmetric case, for all $m\in[1,p+1]$ we define  ${\phi'}_m^i\in U_q(\g_-)$ of weight $-\ve_0-\ve_1+\al_{m}$ by deleting the leftmost copy of $f_{\al_m}$ from $\phi_m^i$. Note that ${\phi'}_{1}^1 ={\phi'}^2_{1}$ and, for even $N$,
${\phi'}_{p+1}^1 ={\phi'}_{p}^1$, ${\phi'}_{p+1}^3 ={\phi'}^3_{p}$. Put
$r_m=2$ for $m=1$ and $r_m=1$ for $m>1$.
\begin{lemma}
For all $m=1,\ldots, p+1$, the vectors $\{{\phi'}^i_m\}_{i=r_m}^{3}\subset [U^-_\l]_{-\ve_0-\ve_1+\al_m}$ are linearly independent.
\end{lemma}
\begin{proof}
One can check that the Gram matrix of the system $\{{\phi'}^i_m v_\la\}_{i=r_m}^3$ with respect to the Shapovalov form
on $M_\la$ is
$$
\left(
\begin{array}{lllll}
& [\la_{1}]_{q}[\la_{0}-\la_1+1]_{q}
& [\la_{1}]_{q}[\la_{0}-\la_1 ]_{q}
\\[1pt]
& [\la_{1}]_{q}[\la_{0}-\la_1 ]_{q}
& [\la_{1}+1]_{q}[\la_{0}-\la_1]_{q}\\
\end{array}
\right),
\quad m=1,
$$
$$
[\la_{1}]_{q}\left(\begin{array}{lll}
[\la_{1}-1]_{q} [\la_{0}-\la_1+2]_{q}
&
[\la_{1}-1]_{q} [\la_{0}-\la_1+1]_{q}
&
[\la_{1}-1]_{q}[\la_{0}-\la_1]_{q}
\\[1pt]
[\la_{1}-1]_{q} [\la_{0}-\la_1+1]_{q}
&
[\la_{1}]_{q}	[\la_{0}-\la_1+1]_{q}
&
[\la_{1}]_{q}[\la_{0}-\la_1]_{q}
\\[1pt]
[\la_{1}-1]_{q}[\la_{0}-\la_1]_{q}
&
[\la_{1}]_{q}[\la_{0}-\la_1]_{q}
&
[\la_{1}+1]_{q}[\la_{0}-\la_1]_{q}
\end{array}
\right),
$$
$m=2, \ldots, n,$ for either parity of $N$. Its determinant is equal to
$$
\begin{array}{lll}
[\la_{0}-\la_1]_{q}[\la_{1}]_{q}[\la_{0}+1]_{q}, \quad m=1,\\[1pt]
[\la_{0}-\la_1]_{q}[\la_{1}]_{q}^3[\la_{1}-1]_{q}[\la_{0}+1]_{q}^2,\quad m=2,\ldots, p+1.\\
\end{array}
$$
It does not vanish for generic $\la$,  hence $\{{\phi'}^i_m v_\la\}_{i=r_m}^3$ are linearly independent. This is also true for
all $\la$, since ${\phi'}^i_m$ are independent of $\la$.
\end{proof}
All $\phi^i_mv_\la$ are annihilated by $e_{\al_0}^2$, as $f_{\al_0}$ enters only once. Therefore their linear combination annihilated
by $e_{\al_i}$, $i>1$, is a generating coefficient.

Present $\C^{d_P^{1}}=\C^{p+1}\oplus \C^{p+1}\oplus \C^{p+1}$ for odd $N$ and
$\C^{d_P^{1}}=\C^{p}\oplus \C^{p+1}\oplus \C^{p}$ for even $N$. Let the upper index of
$(A^i_m)\in \C^{d_P^{1}}$ label the summand in this decomposition while the lower index mark the coordinate within this
summand.

Denote by  $\hat E$ the composition $\C^{d_P^{1}}\to \Phi^{1}\to M_\la$ of linear maps acting by
$(A_m^i)\mapsto \sum_{m,i} A_m^i \phi_m^i =y\mapsto E y$.
It acts by  $\hat E\colon (A_m^i)\mapsto \sum_{i=r_m}^3 B_m^i{\phi'}_m^iv_\la$, where the scalar factors $B_m^i$
are given in Appendix.
Define
$f_{\ve_{0}+\ve_{1}}^{(P)}=\sum_{m,i}A_m^i\phi_m^i$, where $A_m^i$ are as follows:
\be A_{m}^{k } &=&
\begin{cases}
( -1 )^{m+1}  [\la_{1}+P-m]_{q}[\la_{1}+\frac{P}{2}]_{q}, & k=1, \\
( -1 )^{m} 	 (q^{m-\frac{P}{2}}+q^{-m+\frac{P}{2}})_q[\la_{1}+\frac{P}{2} -1]_{q}[\la_{1}+\frac{P}{2}]_{q}, & k=2, \\
( -1 )^{m+1} [\la_{1}+m-1 ]_{q}[\la_{1}+\frac{P}{2} -1]_{q}, & k=3,
\end{cases}
\nn
\ee
for $m=1,\ldots,d^1_P$ apart from
$A_{p}^2,A_{p+1}^2$ for even $N$, which are set to $(-1)^p[\la_{1}+\frac{P}{2} -1]_{q}[\la_{1}+\frac{P}{2}]_{q}$.

\begin{lemma}
Up to a scalar factor, the vector $f_{\ve_{0}+\ve_{1}}^{(P)}v_\la$ is a unique solution of the system $e_{\al_i}f_{\ve_{0}+\ve_{1}}^{(P)}v_\la=0$ for all $i=1,\ldots,p+1$.
Furthermore, $e_{\al_0}f_{\ve_{0}+\ve_{1}}^{(P)}v_\la=[\la_{0}+\la_{1}+P]_{q}f_{2\ve_{1}}^{(P)}v_\la$.
\label{lemma3.8}
\end{lemma}
\begin{proof}
The first part of the statement is proved by a lengthy straightforward calculation, which is omitted here.
Let us prove the second statement.
Observe the identities
$$
\sum_{i=1}^3A_m^i[\la_0-\la_1+3-i]_q=[\la_{0}+\la_{1}+P]_{q}A_m,
$$
which hold for $m=1, \ldots, p+1$, odd $N$,  and for $m=1, \ldots, p-1$, even $N$.
This readily implies the statement for odd $N$:
$$
e_{\al_0}f_{\ve_{0}+\ve_{1}}^{(P)}=\sum_{m=1}^{p+1}\sum_{i=1}^3A_m^ie_{\al_0} \phi_m^iv_\la=\sum_{m=1}^{p+1}\left(\sum_{i=1}^3A_m^i[\la_0-\la_1+3-i]_q\phi_mv_\la\right)
=[\la_{0}+\la_{1}+P]_{q}
f_{2\ve_{1}}^{(P)}.
$$
If $N$ is even, we have also
$$\sum_{i=1}^3A_{p}^i[\la_0-\la_1+3-i]_q+ A_{p+1}^2[\la_0-\la_1+1]=[\la_{0}+\la_{1}+P]_{q}A_{p}.$$
Then, for even $N$,
\be
e_{\al_0}f_{\ve_{0}+\ve_{1}}^{(P)}v_\la&=&\sum_{m=1}^{p}\sum_{i=1}^3A_m^ie_{\al_0} \phi_m^iv_\la+A_{p+1}^2e_{\al_0} \phi_{p+1}^2v_\la=\sum_{m=1}^{p-1}\left(\sum_{i=1}^3A_m^i[\la_0-\la_1+3-i]_q\phi_mv_\la\right)
\nn\\&+&\left(\sum_{i=1}^3A_{p}^i[\la_0-\la_1+3-i]_q+ A_{p+1}^2[\la_0-\la_1+1]\right)\phi_{p}v_\la=[\la_{0}+\la_{1}+P]_{q}
f_{2\ve_{1}}^{(P)},
\nn
\ee
as required.
\end{proof}
\begin{propn}
The vectors $\phi^i_m$ form a basis in $[U^-_\l]_{-\ve_0-\ve_1}$. Up to a scalar factor,
$f_{\ve_{0}+\ve_{1}}^{(P)}v_\la$ is a unique generating coefficient of the weight $\la-\ve_0-\ve_1$.
\end{propn}
\begin{proof}
Observe that $d_P^1$ is equal to the dimension of $[U^-_\l]_{-\ve_0-\ve_1}$, so
we need to prove only linear independence.
Fix a constant $c$ and restrict $\la$ to the hyperplane $\la_1=c$. By Lemma \ref{lemma3.8}, the map $\hat E\colon \C^{d_P^1}\to \Phi^{1}\to M_\la$ is injective for all
$\la$ such that $[\la_{0}+c+2n-1]_{q}\not =0$. Since the map $\C^{d_P^1}\to \Phi^{1}$ is
surjective, the map $E\colon \Phi^{1}\to M_\la$ is injective too. This implies that $\phi^i_mv_\la$
are linearly independent for such $\la$. Since $\phi^i_m$ are independent of $\la_0$, they are linearly
independent at all $\la$ subject to $\la_1=c$, and so are $\phi^i_mv_\la$. As $c$ is arbitrary, the statement
holds true for all $\la$.
\end{proof}
\subsection{The case $l=2$}
\label{Sec_l=2}
In order to relate our calculation to already considered cases $l=0,1$, we enumerate the roots as $\al_{-1},\al_0,\al_1,\ldots, \al_{p+1}$.
We are looking for the generating coefficient of weight $\la-\ve_{-1}-\ve_1$.
It is an element of $M_\la$ satisfying the equations $e_{\al_{-1}}^2y=e_{\al_j}y=0$, $j\geqslant 0$.

Define the element
\be
f_{\ve_{-1}+\ve_{1}}^{(P)}
&= f_{\al_{-1}}f_{\ve_{0}+\ve_{1}}^{(P)}[h_{\ve_0}+h_{\ve_1}+P+1]_{q} - f_{\ve_{0}+\ve_{1}}^{(P)} f_{\al_{-1}}[h_{\ve_0}+h_{\ve_1}+P]_{q} \in U_q(\b_-),
\ee
 of weight $-\ve_{-1}-\ve_0-\ve_1$.
\begin{propn}
The element $f_{\ve_{-1}+\ve_{1}}^{(P)}v_\la\in M_\la$ is a unique generating coefficient
of weight $-\ve_{-1}-\ve_0-\ve_1$.
Furthermore,
$$
e_{\al_{-1}}f_{\ve_{-1}+\ve_{1}}^{(P)}v_\la=[\la_{-1}+\la_{1}+P+1]_{q}f_{\ve_{0}+\ve_{1}}^{(P)}v_\la.
$$

\end{propn}
\begin{proof}
We are looking for the generating coefficient in the form
\be
y
&= \sum_{m,k} (A_{m}^{k1} f_{\al_{-1}} \varphi_{m}^{k} - A_{m}^{k2} \varphi_{m}^{k} f_{\al_{-1}}),
\ee
where $(A^{k1}_m),(A^{k2}_m)\in \C^{d_P^1}$.
Since $f_{\al_{-1}} {\varphi'}_{m}^{k}$ and ${\varphi'}_{m}^{k} f_{\al_{-1}}$ are independent, the conditions
$e_{\al_m}f_{\ve_{-1}+\ve_{1}}^{(P)}v_\la=0$ for positive $m$
give $A_{m}^{kj}=A_{m}^{k}C^j$ for some scalars $C^j$, $j=1,2$. That is,
$y=C^1f_{\al_{-1}}f_{\ve_{0}+\ve_{1}}^{(P)}v_\la- C^2f_{\ve_{0}+\ve_{1}}^{(P)}f_{\al_{-1}}v_\la$.

The coefficients $C^1$, $C^2$ are found from the condition $e_{\al_0}y=\sum_{m=1}^n E_m f_{\al_{-1}} {\varphi}_{m}=0$, where
$E_m$ are equal to
$$
\left( A_{m}^{1} \left[\la_{0}-\la_1+ 2\right]_{q} + A_{m}^{2} \left[\la_{0}-\la_1 + 1\right]_{q}
	+ A_{m}^{3} \left[\la_{0}-\la_1\right]_{q} \right) C^{1} -
$$
$$- \left( A_{m}^{1} \left[\la_{0}-\la_1 + 3\right]_{q} + A_{m}^{2} \left[\la_{0} -\la_1+ 2\right]_{q}
	+ A_{m}^{3} \left[\la_{0}-\la_1+1\right]_{q} \right) C^{2}.
$$
This boils down to $m$ equations $E_m=0$ on $C^i$. One can check that system is consistent and $C^1=[\la_{0}+\la_{1}+P+1]_{q}$, $C^2=[\la_{0}+\la_{1}+P]_{q}$, up to a common scalar
factor.
Thus, $y=f_{\ve_{-1}+\ve_{1}}^{(P)}v_\la$ is a generating coefficient.
\end{proof}
\subsection{Generating coefficients for arbitrary $l\geqslant 0$}

Now we return to the usual enumeration of simple roots, $\al_1,\ldots, \al_n$.
The algebra $\g=\s\o(2l+2+P)$ includes the subalgebra $\s\o(6+P)$ via
the assignment $\al_i\mapsto \al_{l+i}$, i.e.
$$\al_{-1}\mapsto \al_{l-1}, \quad \al_{0}\mapsto \al_{l},\quad \ldots, \quad \al_{p+1}\mapsto \al_{l+p+1}=\al_{n}.$$
Under this embedding, $f_{\ve_{l+1}+\ve_{l+2-i}}^{(P)}(\la)$, $i=1,2,3$, become elements of $U_q(\g_-)$ of weights $-\ve_{l+2-i}-\ve_{l+1}$. The subalgebra  $\s\o(6+P)$ corresponds to already considered case $l=2$

Define an element $f_{\ve_{l-1}+\ve_{l+1}}^{(P)}\in U_q(\b_-)$ by setting
\be
f_{\ve_{l-1}+\ve_{l+1}}^{(P)}
&= f_{\al_{l-1}}f_{\ve_{l}+\ve_{l+1}}^{(P)}[h_{\ve_l}+h_{\ve_{l+1}}+P+1]_{q} - f_{\ve_{l}+\ve_{l+1}}^{(P)}f_{\al_{l-1}}[h_{\ve_l}+h_{\ve_{l+1}}+P]_{q} ,
\ee
so that $f_{\ve_{l-1}+\ve_{l+1}}^{(P)}(\la)$ is indeed the evaluation of $f_{\ve_{l-1}+\ve_{l+1}}^{(P)}$ at the point $\la\in \h^*$.
Observe that
$$
e_{\al_{k}}f_{\ve_{k}+\ve_{l+1}}^{(P)}v_\la=[\la_{k} +\la_{l+1}+P+l-k]_{q}f_{\ve_{k+1}+\ve_{l+1}}^{(P)}v_\la,
$$
once  $k=l-1,l$.
Suppose we have defined $f_{\ve_{k+1}+\ve_{l+1}}$ for some $k\in [1,l-1]$. Then put
\be
f_{\ve_{k}+\ve_{l+1}}^{(P)}
&=& f_{\al_{k}}f_{\ve_{k+1}+\ve_{l+1}}^{(P)}[h_{\ve_{k+1}}+h_{\ve_{l+1}}+P+l-k]_{q}
\nn\\
&-&
f_{\ve_{k+1}+\ve_{l+1}}^{(P)}f_{\al_{k}}[h_{\ve_{k+1}}+h_{\ve_{l+1}}+P+l-k-1]_{q}. \nn
\ee
\begin{propn}
The vectors $f_{\ve_{k}+\ve_{l+1}}^{(P)}v_\la\in M_\la$ satisfy  the equations
\be
e_{\al_j}f_{\ve_{k}+\ve_{l+1}}^{(P)}v_\la&=&\dt_{jk}[\la_{k}+\la_{l+1}+P+l-k]_{q}f_{\ve_{k+1}+\ve_{l+1}}^{(P)}v_\la, \quad k=1,\ldots, l,
\nn\\
e_{\al_{j}}f_{2\ve_{l+1}}^{(P)}v_\la&=&\dt_{j\> l+1}[\la_{l+1}+\frac{P}{2}-1]_{q}f_{\ve_{l+2}+\ve_{l+1}}^{(P)}v_\la,
\ee
where $f_{\ve_{l+2}+\ve_{l+1}}^{(P)}=\phi'_1$.
Then $f_{\ve_{1}+\ve_{l+1}}^{(P)}v_\la$ is a unique generating coefficient of the singular vector in $\C^N\tp M_\la$
of weight $\la-\ve_1-\ve_{l+1}$.
\end{propn}
\begin{proof}
The case of $k=l-1,l,l+1$ has been worked out in Sections \ref{Sec_l=0}-\ref{Sec_l=2}. We suppose that the statement is proved for some $k+1\leqslant l+1$
and prove it for $k$.
Clearly $e_{\al_j}f_{\ve_{k}+\ve_{l+1}}^{(P)}v_\la=0$ for $j>k+1$ by the induction assumption and $j<k$ by construction. The element $f_{\ve_{k+1}+\ve_{l+1}}^{(P)}$ of weight
$-\ve_{k+1}-\ve_{l+1}$ commutes with $e_{\al_{k}}$ modulo $U_q(\b^-)e_{\al_{k-1}}$, which readily implies the formula for $j=k$.
Then the remaining equality $e_{\al_{k+1}}f_{\ve_{k}+\ve_{l+1}}^{(P)}v_\la=0$ easily follows from the induction assumption
$$
e_{\al_{k+1}}f_{\ve_{k+1}+\ve_{l+1}}^{(P)}=[\la_{k+1}+\la_{l+1}+P+l-k-1]_{q}f_{\ve_{k+2}+\ve_{l+1}}^{(P)}v_\la.
$$

Finally, we argue that $f_{\ve_{1}+\ve_{l+1}}^{(P)}v_\la$ does not turn zero for all $\la$. We showed in Sections \ref{Sec_l=0}--\ref{Sec_l=2} that $f_{\ve_{k}+\ve_{l+1}}^{(P)}v_\la\not=0$ for $k=l,l+1,l+2$.
Assuming it is true for all $k\leqslant l$, observe that $f_{\ve_{k}+\ve_{l+1}}^{(P)}$ is a "modified commutator" of $f_{\ve_{k+1}+\ve_{l+1}}^{(P)}$
with $f_{\al_{k}}$ and that $(\al_{k},\ve_{k+1}+\ve_{l+1})\not = 0$.
Further arguments are based on \cite{M_Shap}, Lemma 9.1, and are similar to the proof of Corollary 9.2 therein.
\end{proof}

Next we determine the principal terms of the generating coefficients.
This will be of importance for our further analysis. Observe that
\be
f_{2\ve_{l+1}}^{(P)}v_\la&=&[\frac{P}{2}]_q\psi^{l+1,N-l}v_\la+\ldots,
\nn\\
f_{\ve_{l}+\ve_{l+1}}^{(P)}v_\la&=& [\la_{l+1}+P-1]_{q}[\la_{l+1}+\frac{P}{2}]_{q}\psi^{l,N-l}v_\la+\ldots,
\nn\\
f_{\ve_{m}+\ve_{l+1}}^{(P)}v_\la&=& [\la_{l+1}+P-1]_{q}[\la_{l+1}+\frac{P}{2}]_{q} \prod_{i=m+1}^{l} [\la_{i}+\la_{l+1}+P+l-i+1]_{q} \psi^{m,N-l}v_\la+\ldots,
\nn
\ee
where $m<l$. The omitted terms contain only non-principal monomials.

Now we can express the principal terms of the coefficients $y_i=y_{N-l,i}$ of the singular vector $u_{N-l}$. Introduce scalar coefficients
$c'_i$ via the equality $y_i=c'_i\psi^{i,N-l}v_\la+\ldots$, where the omitted terms do not contain $\psi^{i,N-l}v_\la$. Note that we have exact equality
$y_{i}=c'_{i}\psi^{i,N-l}v_\la$ for $i=l+2,\ldots, N-l$.
Formula (\ref{y_i conat_fact}) can be rewritten as
$$
y_j=(-1)^{\eps_i+1}q^{(\ve_i-\ve_j,\ve_i)}e_{\ve_j-\ve_i}y_i=(-1)^{\eps_i+1}q^{(\ve_i,\ve_i)}e_{\ve_j-\ve_i}y_i
$$
for all $i,j\in [1,N]$ such that $\ve_j-\ve_i\in \Pi^+$. Then
$$
c'_{m}=(-q)^{m-1}[\la_{l+1}+P-1]_{q}[\la_{l+1}+\frac{P}{2}]_{q} \prod_{i=1}^{m-1}[\la_{i}+\la_{l+1}+P+l-i]_{q}\prod_{i=m+1}^{l} [\la_{i}+\la_{l+1}+P+l-i+1]_{q},
$$
$$
c'_{l+1}=(-q)^{l}[\frac{P}{2}]_{q}  \prod_{i=1}^{l}[\la_{i}+\la_{l+1}+P+l-i]_{q},
$$
$$
c'_{l+2}=(-q)^{l+1}[\la_{l+1}+\frac{P}{2}-1]_{q}  \prod_{i=1}^{l}[\la_{i}+\la_{l+1}+P+l-i]_{q},
$$
where $m=1,\ldots, l$. Assuming $\g=\s\o(2n+1)$, we continue as
\be
\begin{array}{rrlr}
c'_{l+2+k}&=&(-q)^kc'_{l+2}, & k=1,\ldots p,
\\
c'_{n+1+k}&=&(-q)^{n-l-1}q^{k-1}c'_{l+2}, &  k=1,\ldots p,
\\
c'_{n+2+p}&=&(-q)^{n-l-1}q^{p}c'_{l+2}[\la_{l+1}]_q.
\end{array}
\label{c'_2n+1}
\ee
For $\g=\s\o(2n)$, we have
\be
\begin{array}{rrll}
c'_{l+2+k}&=&(-q)^kc'_{l+2}, & k=1,\ldots p-1,
\\
c'_{n+1+k}&=&(-q)^{n-l-2}q^kc'_{l+2}, & k=0,\ldots p,
\\
c'_{n+2+p}&=&(-q)^{n-l-2}q^{p+1}c'_{l+2}[\la_{l+1}]_q.
\end{array}
\label{c'_2n}
\ee
We use these formulas in the next section.

\section{Minimal polynomial for $\Q$.}
In this section we deal with two Levi subalgebras, $\l$ and $\hat \l=\h+\s\o(P)\subset \l$. All objects related to $\hat \l$ will be marked with hat. In particular, $\hat M_\la$ is a parabolic Verma module
induced from $U_q(\hat\l+\g_+)$, while $M_\la$ stands for the one
induced from $U_q(\l+\g_+)$,

Given  a weight $\la\in \mathfrak{C}^*_{\hat \l}$ define  $\hat V_i\subset \C^N\tp \hat M_\la$ to be the submodule generated by $\{w_k\tp v_\la\}_{k=1}^i$.
The sequence $\{0\}=\hat V_0\subset  \hat V_1\subset \ldots \subset\hat  V_N $ forms a filtration, $\hat V_\bullet$, of
$\C^N\tp \hat M_\la$. Its graded component $\gr \hat V_j= \hat V_{j}/ \hat V_{j-1}$ is generated by (the image of) $w_j\tp v_\la$;

Now assume that $\la\in \mathfrak{C}^*_{\l}\subset\mathfrak{C}^*_{\hat \l}$. Recall that $\{w_{m_i}\}_{i=1}^{2\ell+3}$ are the highest weight vectors of the irreducible $\l$-blocks in (\ref{l-decomp}).
Since $\Span\{w_{k}\}_{k=m_i}^{m_i-1}=U_q(\l)w_{m_i}$, the image of $\hat V_{k}$ under projection $\C^n\tp  \hat M_\la\to \C^n\tp  M_\la$
coincides with the image $V_{m_i}$  of $\hat V_{m_i}$ for all $k=m_i,\ldots, m_i-1$.
The sequence $\{0\}=V_0\subset  V_1\subset \ldots \subset  V_N $
forms a filtration $V_\bullet$ of $\C^N\tp  M_\la$ with the graded
 module
\be
\gr V_\bullet=(\oplus_{i=1}^{\ell}\gr  V_{i})\oplus\gr V_{l+1} \oplus  \gr V_{l+2}
\oplus \gr V_{N-l}\oplus (\oplus_{i=1}^\ell\gr  V_{{\ell+3+i}}).
\label{gr_V}
\ee
The graded components $\gr V_i= V_{i}/ V_{i-1}$ are labeled with irreducible $\l$-submodules of (\ref{l-decomp}), and
generated by the images of $w_{m_i}\tp v_\la$ carrying the highest weight $\la+\ve_{m_i}$.
\begin{propn}
\label{levi_filtration}
As a filtration of $U_q(\g_-)$-modules, $V_\bullet$ is independent of $\la\in \mathfrak{C}^*_{\l}$.
\end{propn}
\noindent For a proof, see e.g. \cite{AM}.

For generic $\la\in \mathfrak{C}^*_{\l}$, the graded component $\gr V_{m_i}$ is a parabolic Verma module induced from  $U_q(\l)w_{m_i}\subset \C^N$, hence that is true for all $\la$. The operator $\Q$ is scalar on each 
$\gr V_{m_i}$,
which is a cyclic module of highest weight $\la+\ve_{m_i}$.
Therefore (\ref{gr_V}) determines the spectrum of $\Q$ and a polynomial equation on $\Q$. For generic $\la$ this polynomial
is minimal, but may not be so for special values of $\la$. In particular, we are interested in $\la\in \mathfrak{C}^*_{\l,reg'}$.

Suppose that $i\preceq j$ and fix a path from $i$ to $j$ on the Hasse diagram.
We define $\stackrel{\curvearrowright}{\sum}_{m=i}^{\lower6pt\hbox{$\scriptstyle{j}$}}$ as summation   over all nodes $m$ of that path. We shall use it only when it is path-independent.
\begin{propn}[\cite{AM2}]
Suppose that $i,j\in [1,N]$ are such that $i\prec j$.
Then
\be
w_i\tp \psi^{ij}v_\la
&=&(-1)^{j-i+\stackrel{\curvearrowright}{\sum}_{k=i}^{\lower6pt\hbox{$\scriptscriptstyle{j-1}$}} \eps_k}
q^{(\ve_{j}-\ve_i,\ve_j)-\stackrel{\curvearrowright}{\sum}_{k=i+1}^{\lower6pt\hbox{$\scriptscriptstyle{j}$}}(\ve_{k},\ve_{k})} q^{\la_{j}-\la_i}w_{j}\tp v_\la
 \mod  V_{j-1}.
\label{princ_diag}
\ee
If $\psi$ is a Chevalley monomial of weight $\ve_j-\ve_i$ and $\psi \not =\psi^{ij}$, then $w_i\tp \psi v_\la \in  V_{j-1}$.
\label{prop_princ_diag}
\end{propn}
\noindent
It is also convenient to use an equivalent local version of formula (\ref{princ_diag}):
\be
w_i\tp \psi^{ij}v_\la =(-1)^{\eps_i+1}q^{\la_k-\la_i+(\ve_k-\ve_i,\ve_j-\ve_k)}w_{k}\tp \psi^{kj}v_\la \mod  V_{j-1}
\label{principal_diag}
\ee
where $\ve_i-\ve_k =\al\in \Pi^+$ is a positive simple root for some $i,k\in [1,N]$,
and $j \succeq k$. Note that $(\ref{princ_diag})$ holds true for $\g=\g\l(n)$ and $\l=\oplus_{i=1}^{\ell+1}\g\l(n_i)$ via the embeddings $U_q(\g\l(n))\subset U_q(\s\o(N))$,  $\C^n\subset \C^{N}$, of algebras and their natural representations.

We consider yet another system of $\U_q(\g)$-submodules and compare it with $\{V_i\}_{i=1}^{l+3}$.
As we mentioned, for generic $\la$ the tensor product $\C^N\tp  M_\la$ decomposes into the direct sum
$$
\C^N\tp  \hat M_\la=\oplus_{i=1}^{2l+3}  \hat M_{i},\quad \la\in \mathfrak{C}^*_{\hat \l},\qquad
\C^N\tp  M_\la=\oplus_{i=1}^{2\ell+3}  M_{i},\quad \la\in \mathfrak{C}^*_{\l},
$$
where $\hat  M_i$ and $M_i$ are generated by singular vector $\hat u_i$ and, respectively, by the projection $u_{m_i}$
of rescaled  $\hat u_{m_i}$ (which otherwise might turn zero). The left decomposition holds if the Shapovalov forms of $\hat M_\la$ and
all $\hat M_i$ are not degenerate; the same is true for the right decomposition.
The operator $\Q$ is scalar multiple on $\hat  M_i$ and $M_i$ with the eigenvalues $\hat x_i$ and, respectively, $x_{i}= \hat x_{m_i}$.
Denote $ W_{i}=\sum_{k=1}^i  M_{k}$.
For generic $\la$,   $M_i$ is the parabolic Verma modules induced from the corresponding irreducible
$\l$-submodule of $\C^N$. Therefore, it is independent of $\la$ regarded as an $U_q(\g_-)$-module.
\begin{propn}
\label{direct_sum_decomposition}
There is an inclusion $ W_{i}\subset  V_{i}$. Further, $ W_{i}= V_{i}$ if and only if $ W_{i}=\oplus_{k=1}^i  M_{k}$.
Consequently, $W_{i}= V_{i}$ if and only if $ W_{k}= V_{k}$ for all $k\leqslant i$.
\end{propn}
\begin{proof}
The last statement
readily follows from the second.
The inclusion $ W_{i}\subset  V_{i}$ follows from Proposition \ref{prop_princ_diag}.
Since  $ M_{k}$ and $\gr\>  V_{k}$ are cyclic modules of the same highest weight,
either the projection $\pi_k\colon  M_{k}\to \gr\>  V_{k}$ is zero or coincides with $\gr\>  V_k$,
which is the case for generic $\la\in \mathfrak{C}^*_{\l}$.
In particular, $M_k$ is isomorphic to $\gr\>  V_k$ for all $\la$.
Denote $M_{k}'=W_{k-1}\cap M_{k}$. For each $k$ the projection $\pi_k$ factorizes to the
composition
$$M_k\twoheadrightarrow M_k/M_{k}'\simeq W_k/W_{k-1}\twoheadrightarrow W_k/(W_k\cap V_{k-1})\hookrightarrow \gr V_k,$$
where the left and middle arrows are surjective and the right one is injective. As argued, $\pi_k$
is either an isomorphism or $\pi_k=0$. If $M_{k}'=\{0\}$ for all $k\leqslant i$, then,  by ascending induction on $k$, all these maps
are isomorphisms, and $V_k=W_k$ including $k=i$.
Conversely, assuming  $V_i=W_i$, we get $M_{i}'=\{0\}$ and $V_{i-1}=W_{i-1}$. Descending induction on $i$ completes the proof.
\end{proof}
\begin{corollary}
For all $j\in [1,N]$, decomposition $W_j=\oplus_{i=1}^{j}M_i$ holds if and only if
$\pi_{i}(u_{i})\not =0$ for all $i=1,\ldots, j$.
\end{corollary}
In particular, if the eigenvalues $\{x_k\}_{k=1}^N$ are pairwise distinct, the sum $W_{2\ell+3}=\oplus_{k=1}^{2\ell+3}  M_{k}$
is direct, and $W_{2\ell+3}=V_{2\ell+3}=\C^N\tp M_\la$. However, we are interested in the situation
when $x_{\ell+1}=x_{\ell+3}$. To address this case, we need to calculate the $\pi_{\ell+3}(u_{\ell+3})\in \gr V_{\ell+3}$.

Let $C_i$, $i=1,\ldots, 2\ell+3$, be the  scalar coefficient in the presentation $u_i= C_i w_i\tp v_\la \mod V_{i-1}$ and
$\hat C_i$ be similarly defined for $i=1,\ldots, 2l+3$.
Note that the image of $\hat u_i$ may turn zero in $\C^N\tp M_\la$, so $u_i$ is obtained
from $\hat u_{m_i}$ after an appropriated rescaling. This implies that $C_i$ is proportional to $\hat C_{m_i}$
up to a factor turning zero at $\la\in \mathfrak{C}^*_{\l}$.

Our next
goal is to calculate $\hat C_i$ for some $i$ of importance. We do it first for $i={n+1}$ in the case of odd $N$.
Retaining the principal term, we  write
$$
y_{n+1,i}=(-q)^{i-1}\prod_{k=1}^{i-1}[\la_{k}+n-k]\prod_{k=i+1}^{n}[\la_{k}+n-k+1]
\psi^{i,n+1}v_\la +\cdots
$$
\begin{propn}
$
\hat C_{n+1}=\prod_{j=1}^{n}[\la_j+1 + n-j]_q.
$
\end{propn}
\begin{proof}
One can check that
\be
\hat C_{n+1}=\sum_{i=1}^{n+1} q^{i-1}q^{-\la_i+i-n-\dt_{i\>{n+1}}}
\prod_{j=1}^{i-1}[\la_j + n-j]_q
\prod_{j=i+1}^{n}[\la_j + n+1-j]_q
\label{C-gl}
\ee
Replacing $\la_i$ with $\la_i-\la_{n+1}$ one gets the expression, which is
shown in \cite{AM}, Lemma 6.1, to be equal to $\prod_{j=1}^{n}[\la_j-\la_{n+1}+1 + n-j]_q$,
for any $\la_i$, $i=1,\ldots,n+1$. This proves the lemma.
\end{proof}

Next we calculate  $\hat C_{N-l}$. First we assume $l=0$.
The coefficient $\hat C_N$ is $\sum_{i=1}^N c'_ic''_i$, where 
$
c'_{1}=[\frac{P}{2}]_q$, $c'_{2}=-q[\la_{1}+\frac{P}{2}-1]_{q}$,
and $c'_i$  for $i>2$ are given by formulas (\ref{c'_2n+1}) and (\ref{c'_2n}) (one should put $l=0$ there).
The coefficients $c''_i$ are obtained by specialization of (\ref{princ_diag}).
For $N=2n+1$ they are $c''_{N}=1$ and
$$
c''_1=(-1)^{p-1}q^{-2\la_{1}}q^{-2p+1}
,\>
c''_{1+k}=(-1)^{p-1-k}q^{-\la_{1}}q^{-2p+k},
\>
c''_{n+m}=q^{-\la_{1}}q^{-p-1+m},
\>
c''_{N-1}=q^{-\la_{1}},
$$
where $k=1,\ldots, p$, $m=1,\ldots, m+1$.
For $N=2n$, they are $c''_{N}=1$ and
$$
c''_1=(-1)^{p}q^{-2\la_{1}}q^{-2p+2}
,\quad
c''_{1+k}=(-1)^{p-k}q^{-\la_{1}}q^{-2p+1+k}
\quad
c''_{n+k}=q^{-\la_{1}}q^{-p+k},
\quad
c''_{N-1}=q^{-\la_{1}},
$$
where $k=1,\ldots, p$.

\begin{lemma}
In the symmetric case $l=0$, the singular vector $\hat u_{N}$ is equal to $\hat C_N w_{N}\tp v_\la$ modulo $\hat V_{N-1}$,
where $\hat C_N=(-1)^{[\frac{P+1}{2}]}[\la_1+\frac{P}{2}]_q[\la_1+P-1]_q$.
\label{C,l=0}
\end{lemma}
\begin{proof}
The coefficient $(-1)^{[\frac{P+1}{2}]}\hat C_N$ is equal to
$$
q^{-2\la_{1}-2p+1}[\frac{P}{2}]_q+[\la_{1}+\frac{P}{2}-1]_{q}q^{-\la_{1}}\bigl(-\frac{q^{-2p+1}-q}{q-q^{-1}}
+q+\frac{q^{2p+1}-q^{}}{q-q^{-1}}\bigr)+q^{2p+1}[\la_{1}+\frac{P}{2}-1]_{q}[\la_{1}]_q
$$
if $P=2p+1$. For $P=2p$, it is equal to
$$
q^{-2\la_{1}-2p+2}[\frac{P}{2}]_q+[\la_{1}+\frac{P}{2}-1]_{q}q^{-\la_{1}}q\bigl(-\frac{q^{-2p+1}-q}{q-q^{-1}}
+\frac{q^{2p-1}-q^{-1}}{q-q^{-1}}\bigr)+q^{2p}[\la_{1}+\frac{P}{2}-1]_{q}[\la_{1}]_q.
$$
Counting the coefficients before $q^{\pm 2\la_{1}}$ and $\la$-independent terms proves the statement.
\end{proof}

Now consider the general case $l>0$.
\begin{propn}
The singular vector $\hat u_{N-l}$ is equal to $\hat C_{N-l}w_{N-l}\tp v_\la$ modulo $\hat V_{N-1}$,
where
$$
\hat C_{N-l}=(-1)^{[\frac{P+1}{2}]+l}[\la_{l+1}+\frac{P}{2}]_q
 \prod_{j=1 \atop j\not=l+1}^{l+2}
[\la_j+\la_{l+1}+P+1 + l-j]_q.
$$
\end{propn}
\begin{proof}
The second sum in the expansion $\hat u_{N-l}=\sum_{i}^lw_i\tp y_i+\sum_{i=l+1}^lw_i\tp y_i$  can be replaced with
$(-q)^{l} \prod_{i=1}^{l}[\la_{i}+\la_{l+1}+P+l-i]_{q}\hat C_{N-l}w_{N-l}\tp y_{N-l}\mod \hat V_{N-l-1}$, where
the factor before $\hat C_{N-l}$ comes from a different normalization of $c'_{l+1}$ and $c'_1$ in Lemma \ref{C,l=0}.
We have
$$
c'_{i}=(-1)^{[\frac{P+1}{2}]}(-q)^{i-1}\hat C_{N-l}\prod_{j=1}^{i-1}[\la_{j}+\la_{l+1}+P+l-j]_{q}\prod_{j=i+1}^{l} [\la_{j}+\la_{l+1}+P+l-j+1]_{q}
$$
and
$
c''_{i}=(-1)^{[\frac{P+1}{2}]-l+i-1}q^{-P-l+i+}q^{-\la_{i}-\la_{l+1}}
$
for $i=1,\ldots,l$. Note with care that $c''_l=-q^{-2}q^{-\la_l+\la_{l+1}}c''_{l+1}$.
Summing up the products
$$
c'_ic''_{i}=(-1)^{l}q^{i-1}q^{-l+i}q^{-\la_{i}-\la_{l+1}-P}
\hat C_{N-l}\prod_{j=1}^{i-1}[\la_{j}+\la_{l+1}+P+l-j]_{q}\prod_{j=i+1}^{l} [\la_{j}+\la_{l+1}+P+l-j+1]_{q}
$$
from $i=1$ to $i=l$ and adding
$
 (-q)^{l} \prod_{i=1}^{l}[\la_{i}+\la_{l+1}+P+l-i]_{q}\hat C_{N-l}
$
 one gets $\hat C_{N-l} (-1)^{l}$ times the right-hand side of (\ref{C-gl}), where one should replace $n$ with $l$ and $\la_i$ with $\la_i+\la_{l+1}+P$ for $i=1,\ldots, l$. Finally, since $\la_{l+2}=0$, the factor $[\la_{l+1}+P-1]_q$ is included in the
product as $[\la_j+\la_{l+1}+P+1+l-j]_q$, $j=l+2$.
\end{proof}

The operator
$\Q$ satisfies on $\C^N\tp \hat M_\la$ the polynomial equation $\prod_{l=1}^{2l+3}(\Q-\hat x_i)=0$.
When projected to $\End(\C^N\tp M_\la)$, it satisfies the equation $\prod_{l=1}^{2\ell+3}(\Q-x_i)=0$, where
$x_i=\hat x_{m_i}$.
Denote by $\bar C_{\ell+3}$ the product of $\hat x_{l+1}-\hat x_k$ over all $k\leqslant l$ such that $k\not =m_i$,
$i=1,\ldots, \ell$. Put $C_{\ell+3}=\frac{\hat C_{\ell+3}}{\bar C_{\ell+3}}$. Using arguments similar to  \cite{AM}, Lemma 6.6, one can prove that
the image of  $u_{\ell+3}=\frac{1}{\bar C_{\ell+3}}\hat u_{\ell+3}$ in $\C^N\tp M_\la$ is regular in $q$ and $\la\in \mathfrak{C}^*_{\l}$. Then $u_{\ell+3}=C_{\ell+3}w_{\ell+3}\tp v_\la\mod V_{\ell+2}$
is a singular vector. Similarly we define $u_{n+1}$ for the case $N=2n+1$, $P=1$.

\begin{propn}
\label{Filt_V_N-l}
Suppose that $\la\in \mathfrak{C}^*_{\l,'}$ and $q\in \C$ are such that $\{x_i\}_{i=1}^{2\ell+3}-\{x_{\ell+3}\}$ are pairwise distinct.
Then $\C^N\tp M_\la=\oplus_{i=1}^{2l+3} M_{i}$.
\end{propn}
\begin{proof}
All we need to check is that the sum $M_{\ell+1}+ M_{\ell+3}$ is direct. We have
$W_{\ell+2}=\oplus_{i=1}^{\ell+2}M_i$ hence $W_{\ell+2}=V_{\ell+2}$, by Proposition \ref{direct_sum_decomposition}.
Further, $C_{\ell+3}\not =0$ implies $W_{\ell+3}=V_{\ell+3}$, hence
$M_{\ell+1}\cap M_{\ell+3}\subset  W_{\ell+2}\cap M_{\ell+3}=\{0\}$, again by Proposition \ref{direct_sum_decomposition}.
\end{proof}

\begin{corollary}
\label{min_poly_Q}
For  $\la\in \mathfrak{C}^*_{\l,'}$, the operator $\Q\in \End(\C^N\tp M_\la)$ satisfies a polynomial equation
of degree $2\ell +2$ with roots $\{x_i\}_{i=1}^{2\ell+3}-\{x_{\ell+3}\}$.
\end{corollary}
\section{Quantization of borderline Levi classes.}
Fix $\la \in \mathfrak{C}_{\l,reg'}^*$ and define $\mub\in \C^{\ell+2}[\![\hbar]\!]$ by
\be
\mu_i=x_i, \quad i=1,\ldots, \ell+2.
\label{mu_param}
\ee
The eigenvalues of $\Q$ on $\End(\C^{N}\tp M_\la)$  are expressed through $\mub$ by
\be
\label{e_v in hatM_la}
\mu_i,\quad \mu_i^{-1}q^{-2N+2(n_i+1)},\quad i=1,\ldots,\ell,\quad \mu_{\ell+1}=-q^{-N+2}, \quad \mu_{\ell+2}=q^{-N+P},
\ee
cf. (\ref{eigenvalues}).
By construction, $\lim_{\hbar\to0}\mub \in \hat \Mc_K'$.

Define central elements $\tau_k  \in U_q (\g)$ by
\be
\label{q-trace}
\tau_k=\Tr\bigl(q^{2h_\rho}\Q^k\bigr)\in \A,
\ee
where $\rho=\frac{1}{2}\sum_{\al\in \Rm_+}\al=\sum_{i=1}^n(\frac{N}{2}-i)\ve_{i}$ is the half-sum of positive roots.
A module $M$ of highest weight $\la$ determines a one dimensional representation $\chi^\la$ of the center of  $U_q (\g)$,
which assigns a scalar to each  $\tau_k$:
\be
\chi^\la(\tau_k)=
\sum_{\nu} q^{2k(\la+\rho,\nu)-2k(\rho,\ve_1)+k(\nu,\nu)-k}
\prod_{\al\in \Rm_+}\frac{ q^{(\la+\nu+\rho,\al)}-q^{-(\la+\nu+\rho,\al)}}{ q^{(\la+\rho,\al)}-q^{-(\la+\rho,\al)}},
\label{char_V}
\ee
cf. \cite{M2}, formula (24).
The summation is taken over  weights $\nu$ of the module $\C^{N}$.
Restriction of $\la$ to $\mathfrak{C}_{\l,reg'}^*$ makes
the right hand side a function of the vector $\mub$ defined in (\ref{mu_param}). We denote this function by
$\vt_{\nb,q}^k(\mub)$, where $\nb=(n_1,\ldots,n_\ell, 1,p)$ is the integer valued vector of multiplicities.
In the limit $\hbar \to 0$, $\vt_{\nb,q}^k(\mub)$ goes over
into the right hand side of (\ref{tr_cl}), where
$\mu_i=\lim_{h\to 0}q^{2(\la,\ve_{m_i})}$, $i=1,\ldots,\ell$.

In general, $\tau^k \! \mod \hbar$ do not separate classical conjugacy classes of $SO(2n)$.
That is done by an additional invariant which nevertheless turns zero on a class with eigenvalues $\pm 1$.
Its quantum counterpart $\tau^-$ yields $\chi^\la(\tau^-)=\prod^n_{i=1}(q^{2(\la+\rho,\ve_i)}-q^{-2(\la+\rho,\ve_i)})$,
cf. \cite{M2}, Proposition 7.4. It vanishes for $\la\in \mathfrak{C}_{\l,'}^*$  and can be ignored.

Denote by $S\in \End(\C^{N})\tp \End(\C^{N})$ the product of the ordinary flip on $\C^{N}\tp \C^{N}$ and the R-matrix
of the form of \cite{FRT}. It is $U_\hbar(\g)$-invariant, i.e. commutes with $\Delta (x)$ for all $x \in U_\hbar(\g)$.
Let $\kappa\in \End(\C^{N})\tp \End(\C^{N})$ be the one-dimensional projector to the
 trivial  $U_\hbar(\g)$-submodule.
Denote by  $\C_\hbar[O(N)]$ the associative algebra generated by
the matrix entries $A=(A_{ij})_{i,j=1}^{N}\in \End(\C^{N})\tp \C_\hbar[O(N)]$
modulo the relations
\be
S_{12}A_2S_{12}A_2=A_2S_{12}A_2S_{12}
,\quad A_2S_{12}A_2\kappa=q^{-{N}+1}\kappa=\kappa A_2S_{12}A_2.
\label{A-matrix}
\ee
These relations are understood in $\End(\C^{N})\tp \End(\C^{N})\tp \C_\hbar[O(N)]$,
and the indices distinguish the two copies of $\End(\C^{N})$, in the usual
way.

The algebra $\C_\hbar[O(N)]$ is an equivariant  quantization of $\C[O(N)]$. The algebra $\C_\hbar[G]$, $G=SO(N)$, is a quotient of $\C_\hbar[O(N)]$
setting a quantized determinant to $1$. Its explicit form is immaterial, because it is
automatically fixed  by the equations of conjguacy class.
The algebra $\C_\hbar[G]$ can be realized as a $U_\hbar(\g)$-invariant subalgebra in $U_q(\g)$,
with respect to the adjoint action.
The embedding is implemented via the assignment
$$
 \End(\C^{N})\tp \C_\hbar[G]\ni
 A\mapsto \Q\in \End(\C^{N})\tp  U_q(\g).
$$
\begin{thm}
\label{QCC}
Suppose that $\la=\mathfrak{C}_{\l,reg'}^*$ and let $\mub$ be as in (\ref{mu_param}).
The quotient of $\C_\hbar[G]$ by the ideal of relations
\be
\prod_{i=1}^{\ell}(\Q-\mu_i)\times (\Q-\mu_{\ell+1})(\Q-\mu_{\ell+2})\times
\prod_{i=1}^{\ell}(\Q-\mu_i^{-1}q^{-2N+2(n_i+1)})=0,
\label{q-min_pol}
\ee
\be
\Tr_q(\Q^k)=\vt_{\nb,q}^k(\mub)
\label{q-traces}
\ee
is an equivariant quantization of the class $\lim_{\hbar\to 0}\mub \in\hat \Mc_L'$.
It is the image of $\C_\hbar[G]$ in the algebra of endomorphisms
of the $U_q(\g)$-module $M_\la$.
\end{thm}
\begin{proof}
The proof is similar to \cite{M1}, Theorem 10.1. and \cite{M2}, Theorem 8.2.
\end{proof}
\noindent
Theorem \ref{QCC} describes the ideal in $\C_\hbar[G]$.
To describe the ideal in $\C_\hbar[O(N)]$, one should replace $\Q$ with $A$ in (\ref{q-min_pol}) and (\ref{q-traces})
and add the relations  (\ref{A-matrix}).

\appendix
\section{Appendix}
\begin{lemma}
Suppose that $N\geqslant 5$ and $y=\sum_{m=1}^{d_P^0}A_m\phi_m v_\la\in \Phi^{0}$. Then, for all $m=1,\ldots,p+1$, one has $e_{m}y=B_m\phi_m'v_\la$, where the scalar factors $B_m$
are
\be
B_1&=&A_1[\la_1]_q+A_2[\la_1-1]_q,\nn\\
 B_{i}&=& A_{i-1}  + [2]_{q}A_{i} + A_{i+1},\quad i=2,\ldots, d_P^1-1,
\nn\\
B_{p}=B_{p+1}&=& A_{p-1} + [2]_{q}A_{p},
\quad \mbox{ for even } N,
 \nn\\
 B_{p+1}&=& A_{p} + \left(1 + [2]_{q} \right) A_{p+1},
\quad \mbox{ for odd } N.
 \nn
 \ee
\end{lemma}
\begin{proof}
A straightforward calculation.
\end{proof}
\begin{lemma}
Suppose that $y=\sum_{m,i}A_m^i\phi_m^i v_\la\in \Phi^{1}$, where $(A_m^i)\in \C^{d_P^{1}}$. Then, for all $m=1,\ldots,p+1$, one has $e_{\al_m}y=\sum_{i=r_m}^3 B_m^i{\phi'}_m^iv_\la$, where the scalar factors $B_m^i$
are as follows.

\noindent
\noindent
a.1) $P=3$.
$$
B^2_1=A^1_1([\la_1]_q+[\la_1-1]_q)+A^2_1[\la_1]_q, \quad B^3_1=A^3_1([\la_1]_q+[\la_1+1]_q)+A^2_1[\la_1]_q,
$$
a.2) $P=2p+1\geqslant 5$
$$
\begin{array}{lll}
\begin{array}{llll}
B^2_1&=&A_{1}^{1}\left[ \la_{1} \right]_{q} + A_{2}^{1}\left[ \la_{1} - 1 \right]_{q}
	+ A_{1}^{2}\left[ \la_{1} + 1 \right]_{q} + A_{2}^{2}\left[ \la_{1} \right]_{q}, \\
B^3_1&=&A_{1}^{3}\left[ \la_{1} + 1 \right]_{q} + A_{2}^{3}\left[ \la_{1} \right]_{q}, \\
B_m^k&=&A_{m-1}^{k} + A_{m}^{k}\left[ 2 \right]_{q} + A_{m+1}^{k}, & 2\leqslant m\leqslant p,\nn\\
B_{p+1}^i&=&A_{p}^{i} + \left( 1 + \left[ 2 \right]_{q} \right)	A_{p+1}^{i} +  A_{p+1}^{2}, &i=1,3,\nn\\
B_{p+1}^2&=&A_{p}^{2} + A_{p+1}^{2}. 			\nn
\end{array}
\end{array}
$$

\noindent
b.1) $P=4$
$$
\begin{array}{llll}
B^2_1=A^1_1[\la_1]_q+A^2_1[\la_1+1]_q, \quad
B^3_1=A^3_1[\la_1+1]_q+A^2_2[\la_1]_q,
\\
B^2_2=A^1_1[\la_1]_q+A^3_1[\la_1+1]_q,\quad
B^3_2=A^2_1[\la_1]_q+A^2_2[\la_1+1]_q.
\end{array}
$$

\noindent
b.2) $P=2p\geqslant 6$
\be
 B_{i}^k&=& A_{i-1}^k  + [2]_{q}A_{i}^k + A_{i+1}^k, \quad i=1,\ldots, p-1,
\ee
whenever the pair $(i,k)$ is distinct from specified below, in which case   $B_{i}^k$ are
$$
\begin{array}{lll}
\begin{array}{llll}
B_{p-1}^2&=&A_{m-3}^{2} + A_{p-1}^{2}\left[ 2 \right]_{q} + A_{p}^{2} + A_{p+1}^{2},\\
B_{p}^2&=&A_{p-1}^{2} + A_{p}^{2}\left[ 2 \right]_{q},& 			\nn\\
B_{p+1}^2&=&A_{p-1}^{2} + A_{p+1}^{2}\left[ 2 \right]_{q},& 			\nn\\
B_{p}^i&=&A_{p-1}^{i} + A_{p}^{i} \left[2\right]_{q} + A_{p+1}^{2}, &i=1,3,  \\
B_{p+1}^i&=&A_{p-1}^{i} + A_{p}^{i} \left[2\right]_{q} + A_{p}^{2}, &i=1,3.  \\
\end{array}
\end{array}
$$
\end{lemma}
\noindent
This is verified by a straightforward brute force calculation, which is omitted here.

\end{document}